\documentclass[12pt]{amsart}
\usepackage{geometry}                
\usepackage{graphicx}

\usepackage{amsmath}
\usepackage{amsfonts}
\usepackage{enumerate}
\usepackage{amssymb}
\usepackage{epstopdf}
\usepackage{bm}
\usepackage{hyperref}

\newtheorem{theorem}{Theorem}[section]
\newtheorem{proposition}[theorem]{Proposition}
\newtheorem{lemma}[theorem]{Lemma}

\theoremstyle{definition}

\author{Gregory R. Chambers}
\author{Yevgeny Liokumovich}

\begin{document}

\title[Splitting a contraction of a simple curve traversed $m$ times]
{Splitting a contraction of a simple curve traversed $m$ times}

\newcommand{\R}{\mathbb R}

\maketitle

\begin{abstract}
Suppose that $M$ is a $2$-dimensional oriented Riemannian manifold, and let $\gamma$
be a simple closed curve on $M$.  Let $m \gamma$ denote the curve formed by tracing $\gamma$
$m$ times.
We prove that if $m \gamma$ is contractible through curves of length less than $L$, then
$\gamma$ is contractible through curves of length less than $L$.

In the last section we state several open questions
about controlling length and the number of
self-intersections in homotopies of curves on
Riemannian surfaces.
\end{abstract}

\section{Introduction}

The main result of this article extends Theorem 1.2 from \cite{CL}:
\begin{theorem}[\cite{CL} Theorem 1.2]
	\label{thm:original}
	Suppose $M$ is a $2$-dimensional orientable Riemannian manifold (possibly with boundary),
	$\gamma$ is a closed curve on $M$,
	and $2 \gamma$ is the closed curve formed by traversing $\gamma$ twice.  If $2 \gamma$
	can be contracted through curves of length less than $L$,
	then $\gamma$ can also be contracted through curves of length less than $L$.
	
\end{theorem}

In this article, we extend this theorem in the following way:

\begin{theorem}
	\label{thm:main}
	Suppose $M$ is a $2$-dimensional orientable Riemannian manifold (possibly with boundary),
	$\gamma$ is a simple closed curve on $M$,
	and $m \gamma$ is the curve formed by traversing $\gamma$ $m$ times (with $m$ an integer).
	If $m \gamma$ can be contracted through curves of length less than $L$,
	then $\gamma$ can also be contracted through curves of length less than $L$.
\end{theorem}

This new theorem extends the original theorem in that it considers $m$-iterates of $\gamma$ for arbitrary $m$,
however, it only pertains to \emph{simple} closed curves. Given a contraction
of $m \gamma$, we give an explicit algorithm for constructing a 
homotopy of $\gamma$. Moreover, the contraction that we obtain
is in fact an isotopy (except for the final constant curve).

The general structure of proof builds on the ideas in \cite{CL}. As in that article, we prove the theorem
by constructing a graph, and then examining the degree of that graph.  In particular, we look at various ways
of cutting each curve in our homotopy at self-intersection points and rejoining them to form collections of simple
closed curves.  These form the vertices of our graph.  We join a pair of vertices by an edge if
there are isotopies between the collections of simple closed curves represented by the vertices.

We then argue that we can start at a vertex which corresponds to $m$ copies of $\gamma$ and find a path
in this graph to a point where the
number of simple closed curves decreases.  This implies that we have a contraction of one of the original
curves which is isotopic to $\gamma$.  This portion of the argument, as in \cite{CL}, relies on
Euler's handshaking lemma 
from his famous solution of the K\"{o}nigsberg's bridges problem.

This result can be viewed as an effective version of the fact that closed orientable $2$-dimensional manifolds
have no torsion; we cannot find a simple closed curve which is more difficult to contract than the curve formed
by traversing it $m$ times.

In the last section we present some related open problems.

\vspace{2mm}

\textbf{Acknowledgments.}
Both authors would like to express their gratitude to the Government of
Ontario for its support in the form of Ontario Graduate Scholarships.

\section{Preliminaries}

We begin by perturbing the given contraction of $m \gamma$ so that the new contraction still
consists of curves of length less than $L$, but that the initial curve has $m - 1$ self-intersections, as in
Figure \ref*{fig:perturbation}.  Furthermore, we can perform this perturbation so that, if we cut this
curve at every self-intersection, we obtain $m$ simple curves, each of which is isotopic of $\gamma$
through curves of length less than $L$ (see Figure \ref*{fig:first_configuration}).
We may assume that $H$ has already been perturbed in this way.

\begin{figure}
	\centering   
	\includegraphics[scale=0.3]{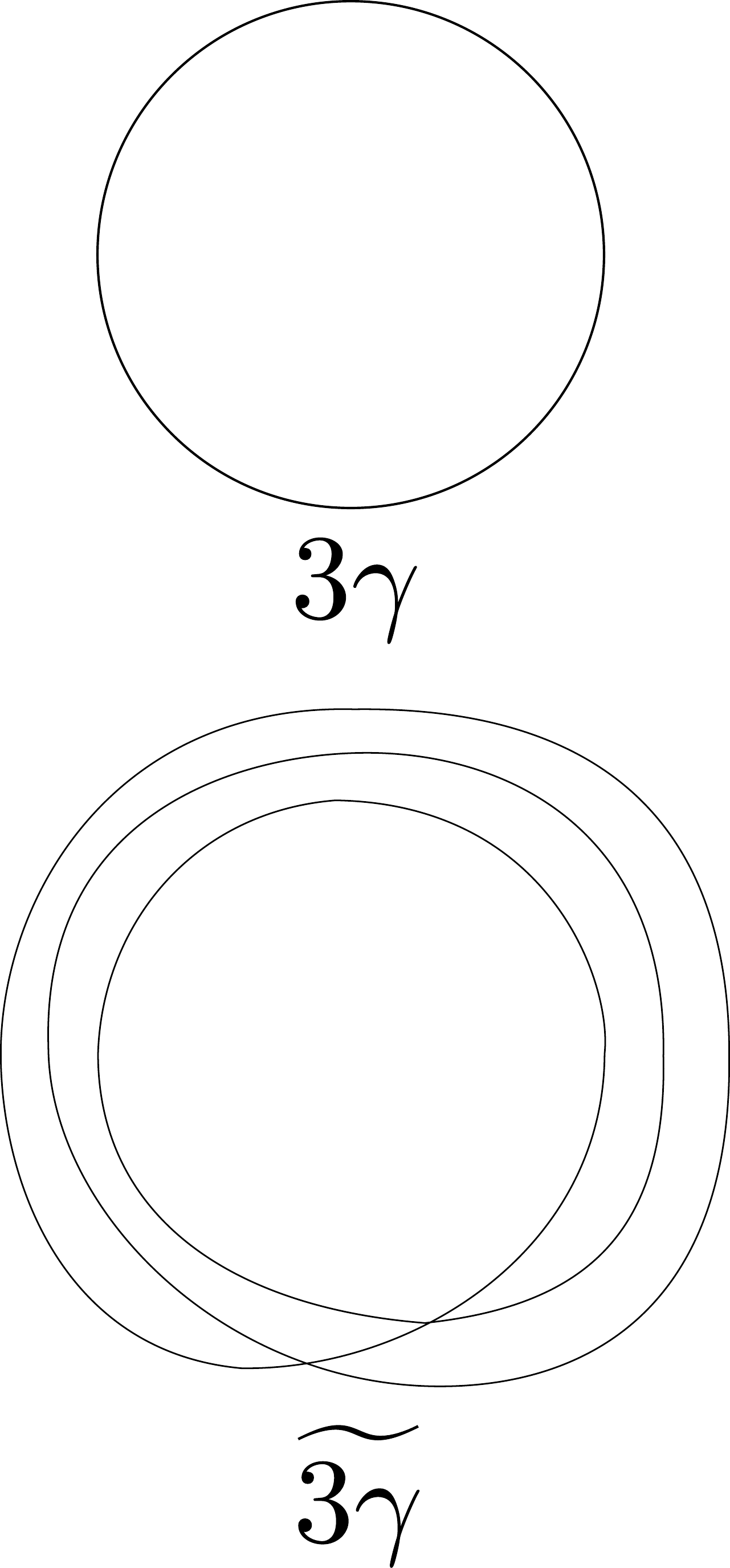}
	\caption{Perturbation of $3 \gamma$ to form $\widetilde{3 \gamma}$} \label{fig:perturbation}
\end{figure}

As in \cite{CL}, we use a parametric form of Thom's Multijet Transversality Theorem as well as
other standard methods (see \cite{A}, \cite{D}, \cite{Bru} and \cite{GG})
to perturb our homotopy $H$ to obtain a homotopy $\tilde{H}$ such that
\begin{enumerate}
		\item	$\tilde{H}(0) = H(0)$
		\item	$\tilde{H}(1)$ is a closed curve which lies in a small disc possessing the
			property that any simple closed curve in this disc of length less than $L$ can
			be contracted in that disc through curves of length less than $L$.
		\item	The length of $\tilde{H}(t)$ is less than $L$ for every $t$.
		\item	For all but finitely many
		$t \in [0,1]$, $H(t)$ is an immersed curve 
		with finitely many transverse self-intersections and
		no triple points.
		\item	There exists a finite number of times $t_0 < \dots < t_n \in [0,1]$
			with $t_0 = 0$ and $t_n = 1$ such that exactly one Reidemeister
			move occurs between $t_i$
			and $t_{i+1}$ for every $1 \leq i < n$. For every point $t_i$,
			$t_i$ satisfies $(4)$.
\end{enumerate}

The Reidemeister moves that we refer to are so named because of their obvious
analogy to the corresponding moves in knot theory (see \cite{K}).  There are three types of them,
and they are shown in Figure \ref*{fig:Reidemeister_moves}.  Note that since $H(0)$
contains self-intersections and $H(1)$ is a simple closed curve, there must be at least
one Reidemeister move; as such, $n \geq 1$.  For the remainder of this article, we will
assume that our homotopy has already been put into this form.

\begin{figure}
   \centering   
    \includegraphics[scale=0.5]{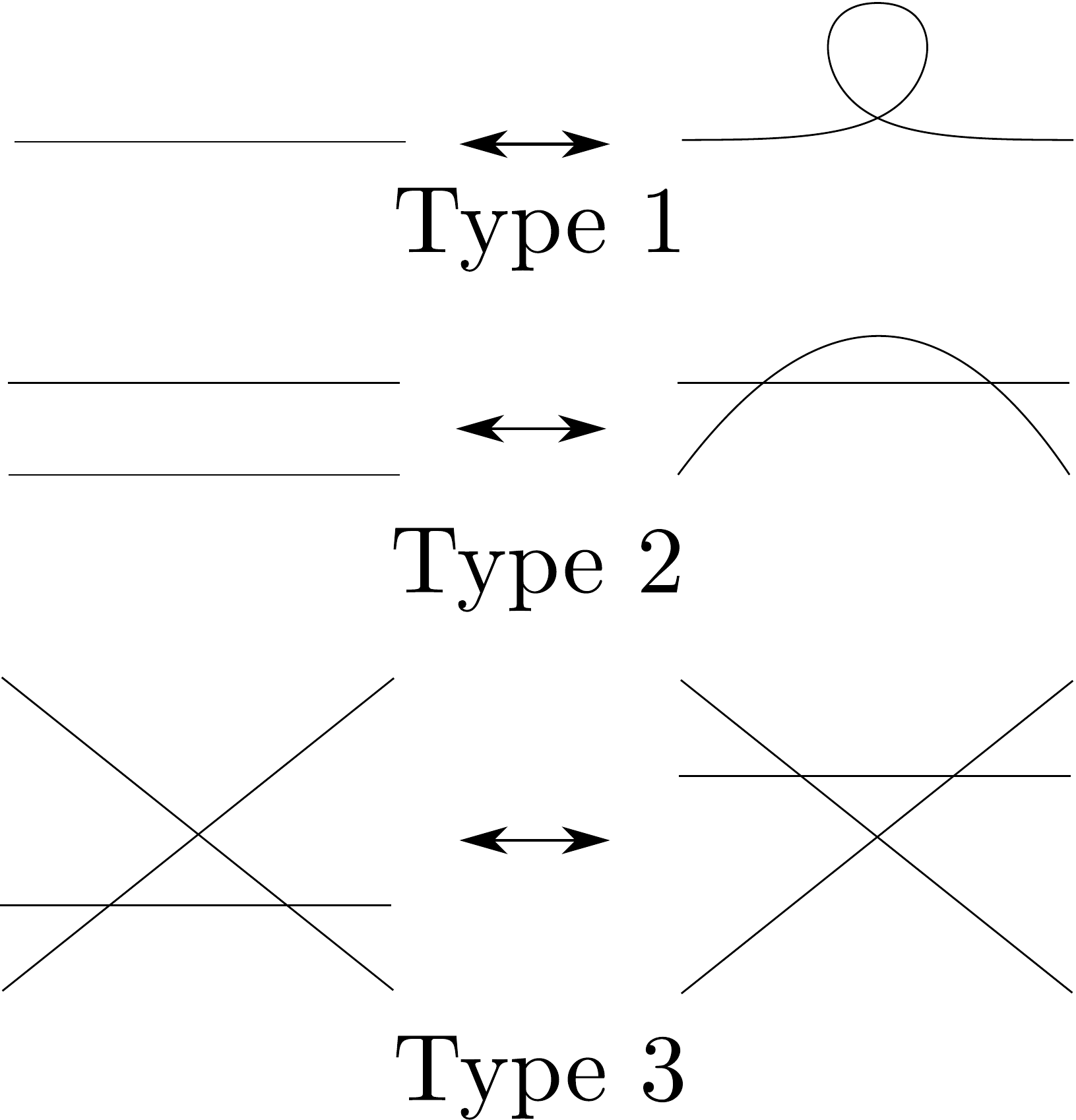}
    \caption{The 3 types of Reidemeister moves} \label{fig:Reidemeister_moves}
\end{figure}

\section{Definition of the graph $\Gamma$}

For each $t_i$, consider $H(t_i)$. Let the self-intersections
of $H(t_i)$ be $s_1, \dots, s_{j}$. For each $s_i$, we can find a
small open ball $B$ centered at $s_i$ so that $H(t_i) \cap B$
consists of two arcs that intersect transversally.  If we choose
an open ball $B'$ centered at $s_i$ which is much smaller than $B$,
then $H(t_i) \cap B \setminus B'$ consists of four open arcs which
have no self-intersections.  There are exactly two ways of connecting
these four arcs inside of $B'$ to obtain two arcs which do not self-intersect.  These
are shown in Figure \ref*{fig:resolution_self_intersection}.
Furthermore, if we choose $B'$ to be small enough, then we can perform either
of these surgeries so that the length of the resulting collection of arcs
does not exceed $L$.  
As a result of these surgeries we will obtain a collection 
of at most $j+1$ closed curves.

\begin{figure}
	\centering   
	\includegraphics[scale=0.25]{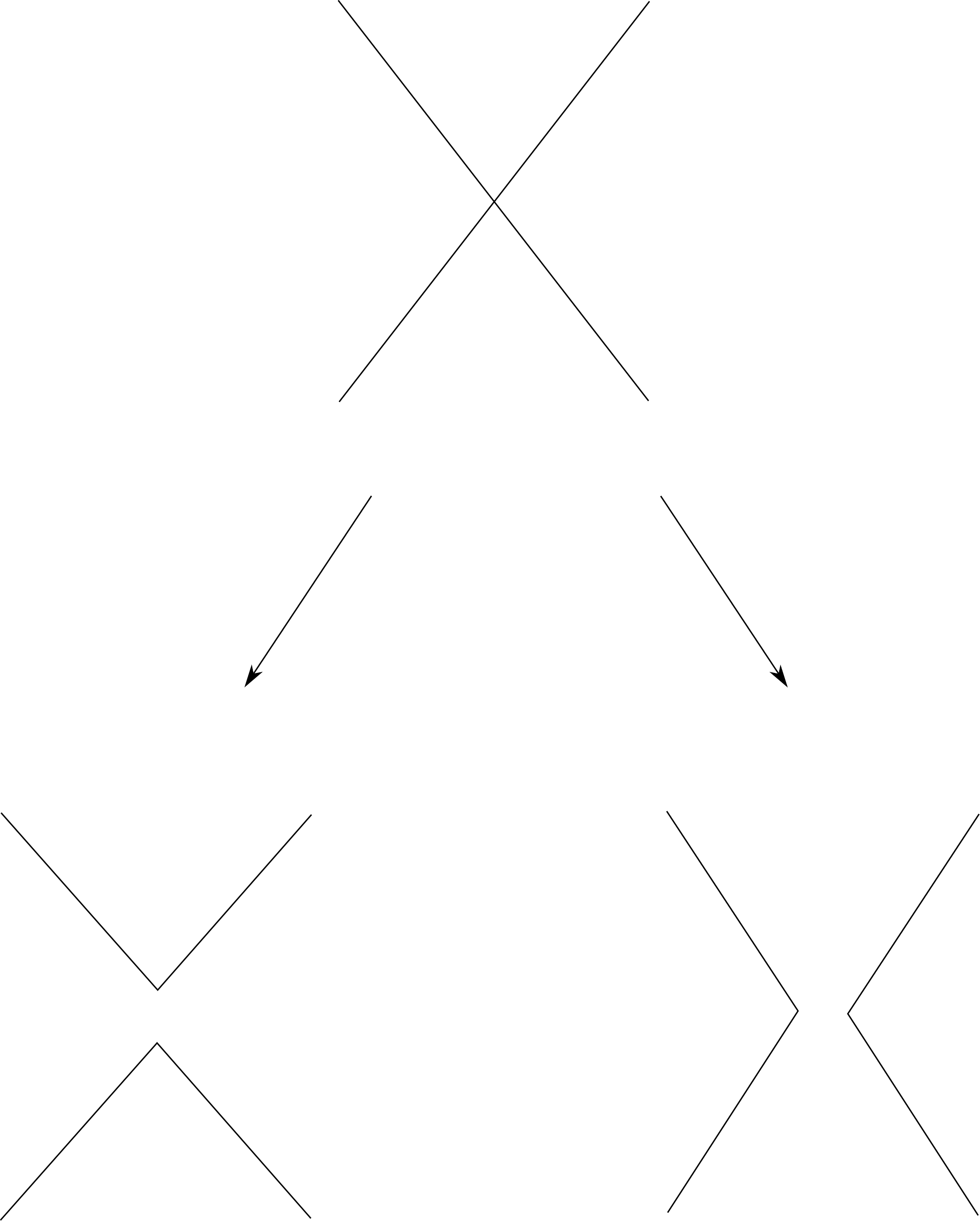}
	\caption{Possible resolutions of a self-intersection}
	\label{fig:resolution_self_intersection}
\end{figure}

As in \cite{CL}, we will define a graph $\Gamma$ whose properties
will allow us to prove the theorem.

\noindent {\bf{Vertices}} We begin by defining the vertices.
We will define $n + 1$ sets of vertices; their disjoint union will form the 
set of vertices of $\Gamma$.  These sets will be denoted as $V_0, \dots, V_n$.
For every $i$ with $0 \leq i \leq n$, consider all
collections of arcs formed by cutting every self-intersection of $H(t_i)$
in each of the two possible ways shown in Figure \ref*{fig:resolution_self_intersection}.
If there are $j$ self-intersections, then there are $2^j$ such collections.
For each collection of simple closed curves obtained from a cutting, we add a vertex to $V_i$. As previously stated, the
vertices of $\Gamma$ are $\sqcup_{i=0}^n V_i$.

\noindent {\bf{Edges}} We now proceed to adding edges.  Every edge in $\Gamma$ is either between a vertex
$v$ in $V_i$ and a vertex $w$ in $V_{i+1}$, or two vertices $v$ and $w$ in $V_i$.
Fix a vertex $v$ in $V_i$, and let $R_1$ be the Reidemeister move between $t_{i-1}$
and $t_i$ (if $i = 0$, then we skip this step).
We now add an edge from $v$ to a vertex $w$ based on the type of move which $R_1$ corresponds to.
Let $C_v$ and $C_w$ be the collections of simple closed curves which $v$ and $w$ respectively
represent.

If $R_1$ is of Type 1, and if $C_1$ and $C_2$ locally correspond to diagrams
which are linked in Figure~\ref*{fig:type_1_resolution}, then we add
an edge between $v$ and $2$.
		
\begin{figure}
	\centering   
	\includegraphics[scale=0.5]{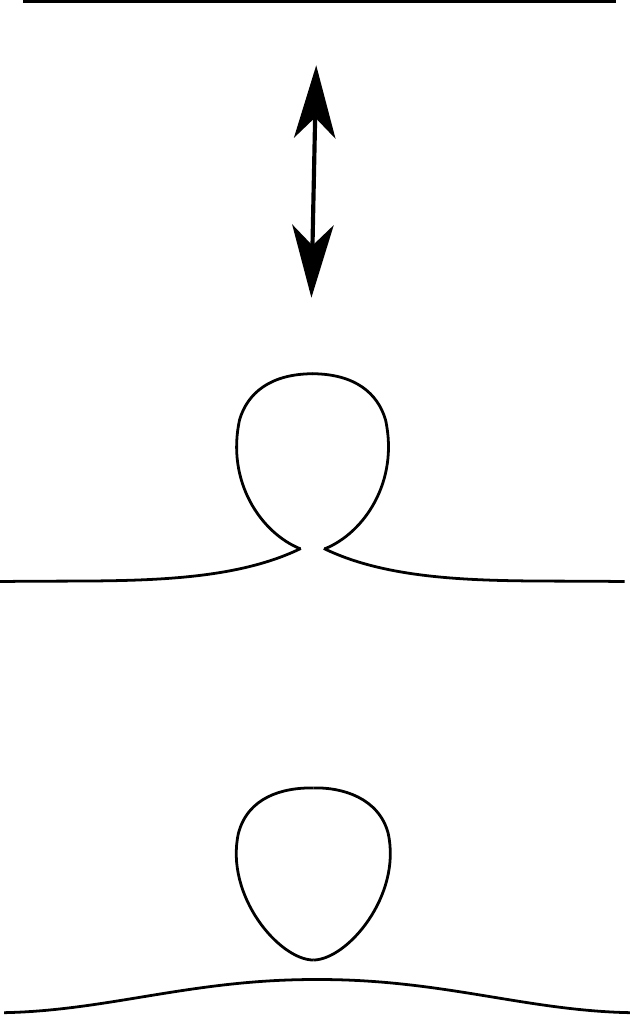}
	\caption{Possible edges resulting from a Type 1 move}
	\label{fig:type_1_resolution}
\end{figure}

If $R_1$ is of Type 2, and if $C_1$ and $C_2$ locally correspond to diagrams
which are linked in Figure~\ref*{fig:type_2_resolution}, then we
add an edge between $v$ and $w$.

\begin{figure}
	\centering   
	\includegraphics[scale=0.5]{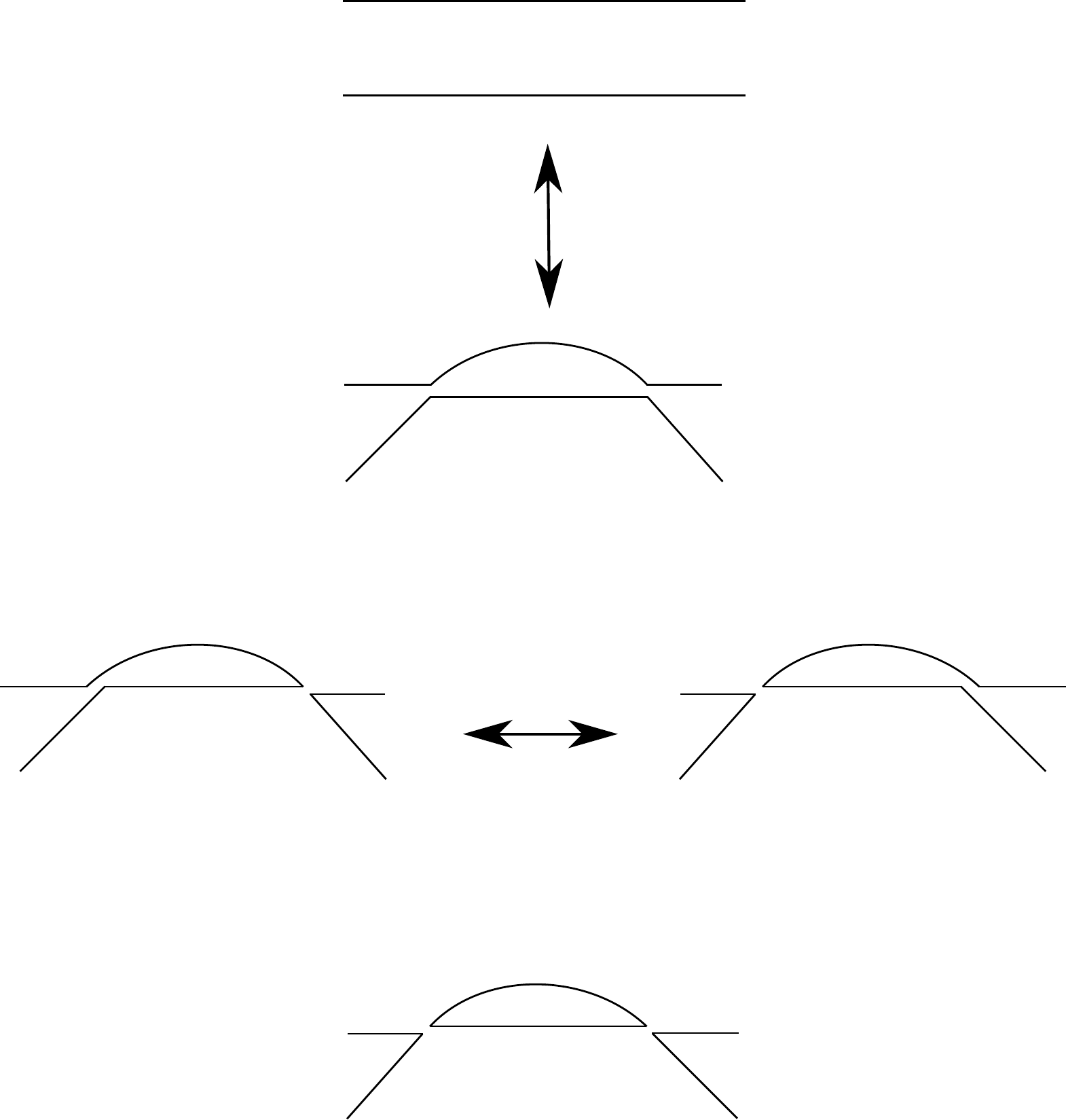}
	\caption{Possible edges resulting from a Type 2 move}
	\label{fig:type_2_resolution}
\end{figure}
			
If $R_1$ is of Type 3, and if $C_1$ and $C_2$ locally correspond to diagrams
which are linked in Figure~\ref*{fig:type_3_resolution}, then we
add an edge between $v$ and $w$.

\begin{figure}
	\centering   
	\includegraphics[scale=0.5]{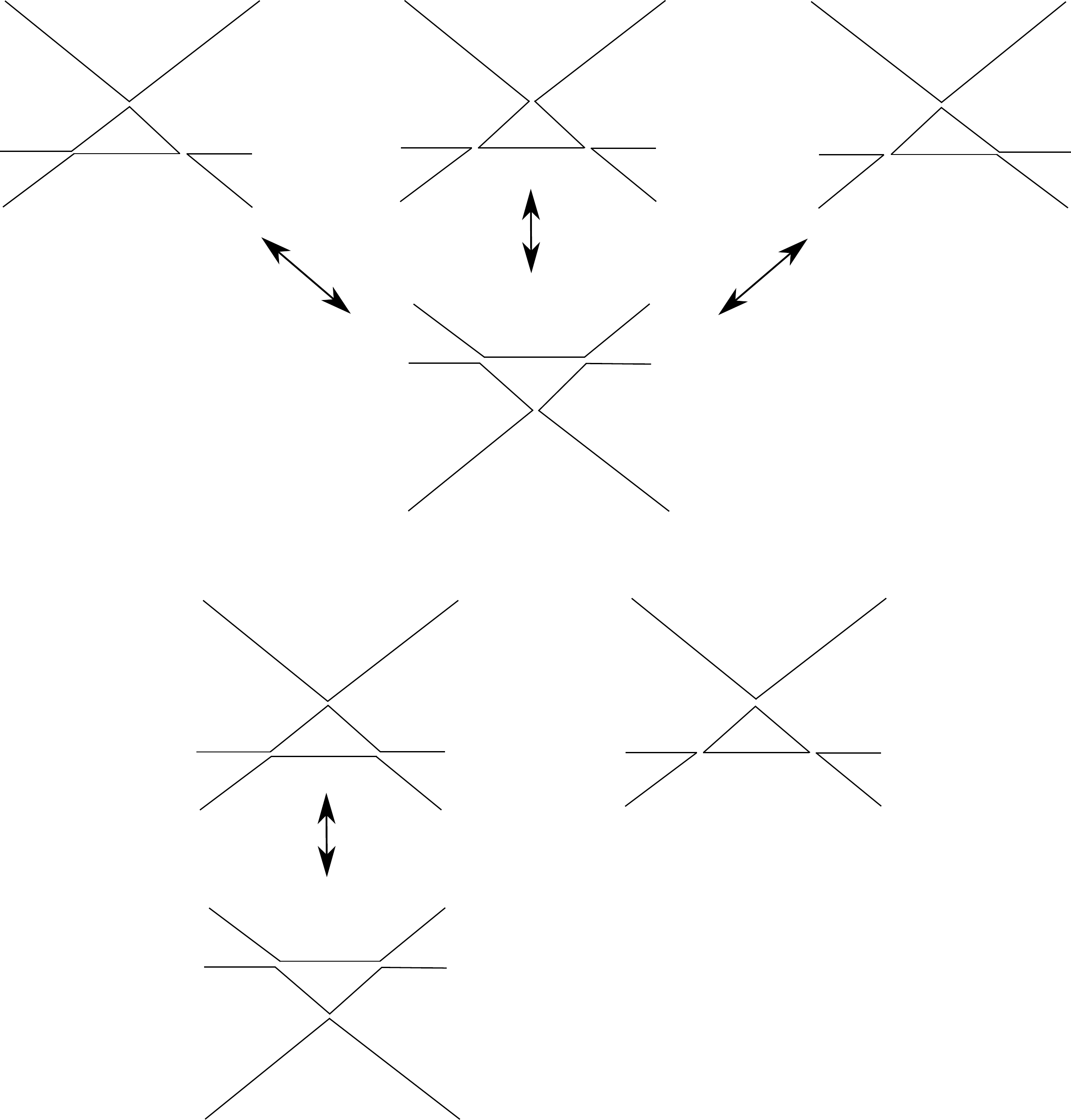}
	\caption{Possible edges resulting from a Type 3 move} \label{fig:type_3_resolution}
\end{figure}

Let $R_2$ be the Reidemeister move between
$t_i$ and $t_{i+1}$ (if $i = n$,
then we skip this step).  We now follow the identical procedure
as we followed with $R_1$ to add edges to $\Gamma$.  Note that if we add
an edge between $v$ and a vertex $w$ in the first step, and we add an
edge between $v$ and $w$ in this second step, then $\Gamma$ has two
edges between $v$ and $w$.

\section{Proof of Theorem~\ref*{thm:main}}

\begin{lemma}
	\label{lem:edge_isotopy}
	If vertices $v$ and $w$ in $\Gamma$ that correspond to collections of curves $C_v$
	and $C_w$ are connected by an edge, then there is
	a bijection between the curves in $C_v$ and the curves in $C_w$
	such that there is an isotopy between each pair of curves that
	are obtained from this bijection.
\end{lemma}
\begin{proof}
	This is proved on a case-by-case basis.  Each edge $e$ is added
	as a result of a Reidemeister move $R$.  In every case,
	we observe that the number of simple closed curves does not change.
	The fact that the curves are isotopic in each case is clear from Figures
	\ref*{fig:type_1_resolution}, \ref*{fig:type_2_resolution} and
	\ref*{fig:type_3_resolution} (see the proof of 
	Theorem 1.2 in~\cite{CL}).
\end{proof}

\begin{lemma}
	\label{lem:degree}
	For every vertex $v$ in $\Gamma$, if $v$ has odd degree, then at least one of the following
	properties is true:
	\begin{enumerate}
		\item	$v$ lies in $V_0$.
		\item	$v$ lies in $V_n$.
		\item	There is a simple closed curve in the collection of curves which corresponds
			to $v$ that is contractible through simple closed curves of length less than $L$.
	\end{enumerate}
\end{lemma}
\begin{proof}
	Assume that there is a vertex $v$ that has none of these properties, but which has odd
	degree.  We have that $v \in V_i$ with $i \neq 0$ and $i \neq n$.  Let $R_1$ be the Reidemeister move
	between $t_{i-1}$ and $t_i$, and let $R_2$ be the Reidemeister move between
	$t_i$ and $t_{i+1}$. The edges with endpoint $v$ are produced from $R_1$ and from $R_2$.
	Looking at Figures~\ref*{fig:type_1_resolution}, \ref*{fig:type_2_resolution}, and
	\ref*{fig:type_3_resolution}, we see that in each case $R_1$ produces either
	$0$, $1$, or $3$ edges.  If it produces $0$ edges, then $v$ must locally look like
	one of the diagrams in Figure~\ref*{fig:zero_edges} depending on the type of $R_1$.
	In each case, the simple closed curve shown in each diagram can be contracted to a constant
	curve within a very small disc; as such, it can be contracted through curves of length
	less than $L$.  Since $v$ does not have this property, $R_1$ must add $1$ or $3$ edges to $v$.

	\begin{figure}
		\centering   
		\includegraphics[scale=0.5]{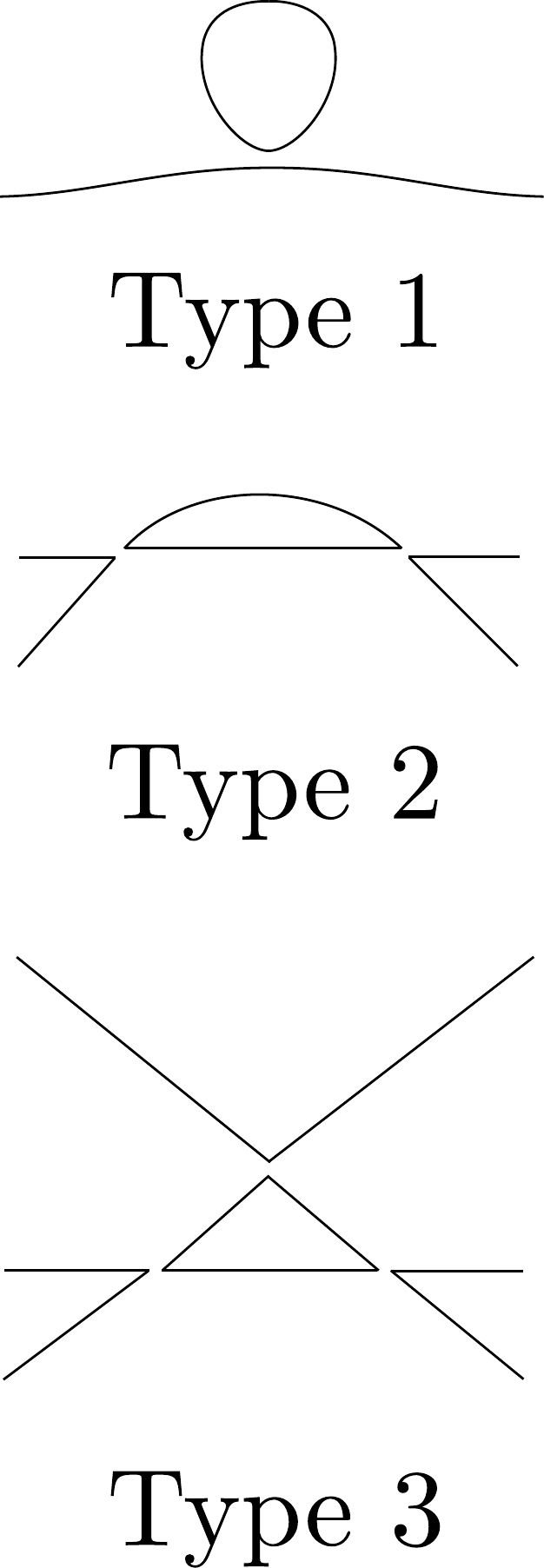}
		\caption{The possible diagrams that correspond to vertices with $0$ edges}
		\label{fig:zero_edges}
	\end{figure}

	The identical analysis for $R_2$ demonstrates that it too contributes $1$ or $3$
	edges to $v$.  As such, the number of edges with endpoint $v$ is $2$, $4$, or
	$6$, proving that it has even degree.

\end{proof}

We will look at a certain subgraph $\Gamma'$ of $\Gamma$.  This subgraph is defined
as the connected component of $\Gamma$ which contains a particular vertex $v^* \in V_0$.
$H(0)$ has $m-1$ self-intersections; if we make the correct choice for resolving each
vertex, then we obtain $m$ simple closed curves, each of which is isotopic to $\gamma$
through curves of length less than $L$. $v^*$ is the vertex in $\Gamma$ that corresponds to
this collection of simple closed curves.  See Figure~\ref*{fig:first_configuration}
for an example of such a collection of curves for $m = 3$.

\begin{figure}
	\centering   
	\includegraphics[scale=0.3]{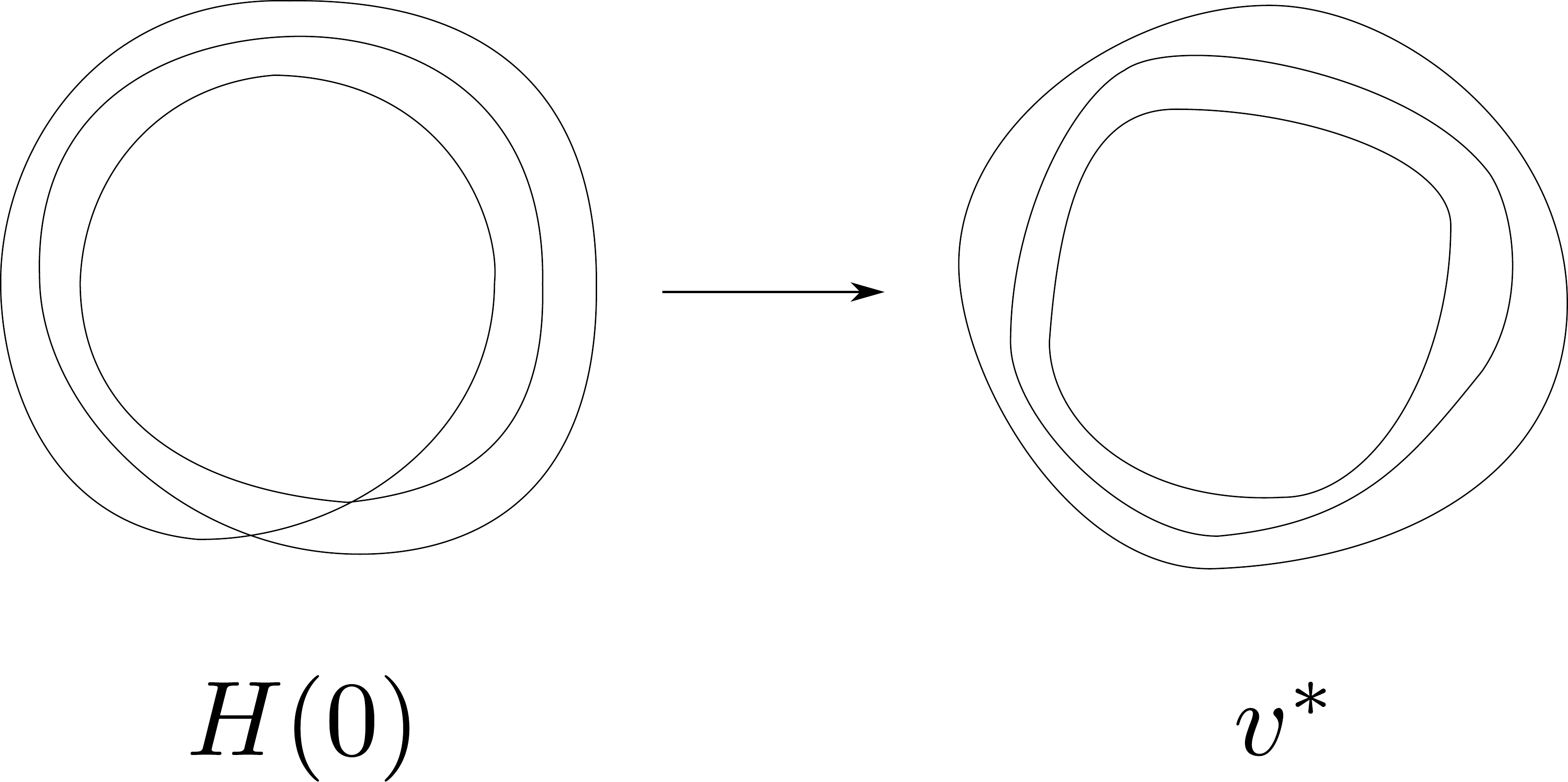}
	\caption{The redrawing which corresponds to the first vertex in our path}
	\label{fig:first_configuration}
\end{figure}

\begin{proposition}
	\label{prop:path}
	For any vertex $w$ in $\Gamma'$, for any simple closed curve $\alpha$ in
	the collection of curves which corresponds to $w$, $\alpha$ is
	isotopic to $\gamma$ through curves of length less than $L$.
\end{proposition}
\begin{proof}
	Since $w$ is in $\Gamma'$, there is a path in $\Gamma$ from $v^*$ to $w$.
	Lemma~\ref*{lem:edge_isotopy} implies that $\alpha$ is isotopic to one
	of the simple closed curves in the collection corresponding to $v^*$
	through curves of length less than $L$.  Each of the simple closed
	curves in this collection is isotopic to $\gamma$ through
	curves of length less than $L$, completing the proof.
\end{proof}

We can now prove Theorem~\ref*{thm:main}.
	
\begin{proof}[Proof of Theorem~\ref*{thm:main}]
	Suppose that $v^*$ is of even degree.  Let $R$ correspond to the 
	Reidemeister move between $t_0$ and $t_1$.  The only edges
	that have $v^*$ as an endpoint come from $R$, and so $R$ must add $0$ edges to $v^*$.
	From the discussion in the proof of Lemma \ref*{lem:degree} and
	Figure~\ref*{fig:zero_edges}, we have that one of the simple closed curves
	in the collection corresponding to $v^*$ is contractible through
	simple closed curves of length less than $L$.  Each of the curves in this
	collection is isotopic to $\gamma$ through curves of length less than $L$,
	completing the proof.

	Suppose now that $v^*$ is of odd degree.  By Euler's handshaking lemma,
	$\Gamma'$ must contain another vertex $w \neq v^*$ that also has odd degree.
	By Lemma~\ref*{lem:degree}, $w \in V_0$, $w \in V_n$, or one of the simple
	closed curves in the collection that corresponds to $w$ is contractible through
	curves of length less than $L$.  We proceed on a case-by-case basis:

	\noindent{\bf{Case 1:} \bm{$w \in V_0$}} \qquad This case cannot occur.
				    Assume that $w \in V_0$.  Then,
				    by Lemma~\ref*{lem:edge_isotopy}, $w$ and $v^*$
				    must correspond to collections consisting of the same
				    number of simple closed curves.  However,
				    since $H(0)$ has $m-1$ self-intersections,
				    there is exactly one resolution of these self-intersections
				    which produces $m$ simple closed curves, and this
				    resolution corresponds to $v^*$.  Hence, $w$ must correspond
				    to a collection of strictly less than $m$ simple closed curves,
				    and so we obtain a contradiction.

	\noindent{\bf{Case 2:} \bm{$w \in V_n$}} \qquad Choose a simple closed curve $\alpha$ that
				    is in the collection
				    of simple closed curves which corresponds to $w$.
				    By Proposition~\ref*{prop:path}, $\alpha$ is isotopic to
				    $\gamma$ through curves of length less than $L$. Since $\alpha$ is composed
				    of arcs of $H(1)$ and $H(1)$ lies in a disc possessing the property
				    that any simple closed curve of length less than $L$ can be contracted
				    through simple closed curves of length less than $L$, $\alpha$
				    is contractible through simple closed curves of length less than $L$.

	\noindent{\bf{Case 3:} \bm{$w \not \in V_0$} \bf{and} \bm{$w \not \in V_n$}} \qquad
				    In this case,
				    by Lemma~\ref*{lem:degree}, there is a simple closed
				    curve $\alpha$ which is in the collection of simple closed
				    curves that corresponds to $w$, and which has the property
				    that it can be contracted through simple closed curves
				    of length less than $L$. By Proposition~\ref*{prop:path},
				    $\alpha$ is isotopic through curves of length less than $L$
				    to $\gamma$, completing the proof.
\end{proof}

\section{Open problems}

Several natural open questions arise from the problems
considered in \cite{CL}, in \cite{CL2}, and in this paper.

The first one is whether Theorem 1.2 holds without the
assumption that $\gamma$ is simple. 

\vspace{2mm}

\noindent \textbf{Conjecture 1 }	Suppose $M$ is a $2$-dimensional orientable Riemannian manifold, $\gamma$ is a closed curve on $M$,
	and $m \gamma$ is the curve formed by traversing $\gamma$ $m$ times.
	If $m \gamma$ can be contracted through curves of length less than $L$, then $\gamma$ can also be contracted through curves of length less than $L$.
	
\vspace{2mm}

One may try to approach this problem using the methods of this paper
by perturbing $m \gamma$ and cutting it at the self-intersections points 
to obtain $m-1$ copies of $\gamma$. One would then like to define
a graph of resolutions similar to the ones constructed in this paper,
in \cite{CL}, or in \cite{CL2}. Unfortunately, the most straightforward
choices lead to graphs with many vertices of odd degree. 
However, it seems to us that one may still be able to obtain
a contraction of $\gamma$ by studying some finer properties
of these graphs and, possibly, by using some additional surgeries
(cf. the surgeries on the homotopies of curves in \cite{CL2}).

Defining such a graph and studying its properties
may lead to a generalization of 
Theorem 1.1 from \cite{CL}. In \cite{CL} the authors proved
that, given a contraction of a simple closed curve through curves 
of length less than $L$, one can find a contraction through simple curves 
satisfying the same length bound. This leads to the following conjecture.

\vspace{2mm}

\noindent \textbf{Conjecture 2} 	Let $\gamma$ be a closed curve with
$k$ self-intersections, and suppose that $\gamma$ can be contracted through curves of length less than $L$.
$\gamma$ can then be contracted through curves of length less than $L$ with at most $k$ self-intersections.

\vspace{2mm}

If Conjecture 2 holds, then it implies an a priori somewhat stronger
statement that the contraction of $\gamma$ can be chosen so
that the number of self-intersections decreases monotonically 
(in other words, intersections are only destroyed in the homotopy
and never created).

Finally, it is of interest whether parametric versions of these results
hold. 

\vspace{2mm}

\noindent \textbf{Question 3} Let $\{\gamma^{s} \}$ be a 
$k-$parameter family of closed curves 
on an orientable Riemannian surface and suppose that
the family $\{2 \gamma^{s} \}$ can be continuously deformed
to a family of points through 
curves of length less than $L$. Does there exist
a similar deformation of $\{\gamma^{s} \}$ through curves
of length less than $L$?

\vspace{2mm}

This problem is closely related to a question of N. Hingston and 
H.-B. Rademacher (see \cite{HR} and \cite{BM}) about min-max levels 
of multiples of homology classes
of the loop space of a Riemannian sphere.

\begin{tabbing}
\hspace*{7.5cm}\=\kill
Gregory R. Chambers                 \> Yevgeny Liokumovich\\
Department of Mathematics           \> Department of Mathematics\\
University of Chicago               \> Imperial College London\\
Chicago, Illinois 60637             \> London SW7 2AZ\\
USA                                 \> United Kingdom\\
e-mail: chambers@math.uchicago.edu  \> e-mail: y.liokumovich@imperial.ac.uk\\
\end{tabbing}

\end{document}